\documentclass[oneside,english,reqno]{amsart}
\usepackage[T1]{fontenc}
\usepackage[latin9]{inputenc}
\usepackage{babel}
\usepackage{units}
\usepackage{mathtools}
\usepackage{amstext}
\usepackage{amsthm}
\usepackage{amssymb}
\usepackage[unicode=true,pdfusetitle,
 bookmarks=true,bookmarksnumbered=false,bookmarksopen=false,
 breaklinks=false,pdfborder={0 0 1},backref=false,colorlinks=false]
 {hyperref}

\makeatletter
\theoremstyle{definition}
\newtheorem*{defn*}{\protect\definitionname}
\theoremstyle{remark}
\newtheorem*{rem*}{\protect\remarkname}
\theoremstyle{plain}
\newtheorem*{thm*}{\protect\theoremname}
\theoremstyle{plain}
\newtheorem{thm}{\protect\theoremname}
\theoremstyle{remark}
\newtheorem*{claim*}{\protect\claimname}

\usepackage{calrsfs}
\usepackage{graphicx}
\usepackage{color}
\usepackage{perpage}
\usepackage{upquote}
\usepackage{bbm}
\usepackage[shortlabels]{enumitem}
\usepackage{stmaryrd} 
\usepackage{extarrows} 
\usepackage{romanbar}

\MakePerPage{footnote}
\gdef\SetFigFontNFSS#1#2#3#4#5{} 
\gdef\SetFigFont#1#2#3#4#5{} 
\def\clap#1{\hbox to 0pt{\hss#1\hss}}

\definecolor{myblue}{rgb}{0.09,0.32,0.44} 
\hypersetup{pdfborder={0 0 0},pdfborderstyle={},colorlinks=true,linkcolor=myblue,citecolor=myblue,urlcolor=blue}

\theoremstyle{remark}
\newtheorem*{qst*}{Question}
\newtheorem*{rmrks*}{Remarks}

\newlength{\tempindent}
\newcommand{\lazyenum}{
\setlength{\tempindent}{\parindent}
\begin{enumerate}[leftmargin=0cm,itemindent=0.7cm,labelwidth=\itemindent,labelsep=0cm,align=left,label=\arabic*)]
\setlength{\parskip}{\smallskipamount}
\setlength{\parindent}{\tempindent}
}

\makeatletter
\def\moverlay{\mathpalette\mov@rlay}
\def\mov@rlay#1#2{\leavevmode\vtop{%
   \baselineskip\z@skip \lineskiplimit-\maxdimen
   \ialign{\hfil$\m@th#1##$\hfil\cr#2\crcr}}}
\newcommand{\charfusion}[3][\mathord]{
    #1{\ifx#1\mathop\vphantom{#2}\fi
        \mathpalette\mov@rlay{#2\cr#3}
      }
    \ifx#1\mathop\expandafter\displaylimits\fi}
\makeatother

\makeatletter
\ifdefined\andify
\renewcommand{\andify}{%
  \nxandlist{\unskip, }{\unskip{} \@@and~}{\unskip{} \@@and~}}
\def\author@andify{%
  \nxandlist {\unskip ,\penalty-1 \space\ignorespaces}%
    {\unskip {} \@@and~}%
    {\unskip \penalty-2 \space \@@and~}%
}
\let\@wraptoccontribs\wraptoccontribs
\@namedef{subjclassname@2020}{%
  \textup{2020} Mathematics Subject Classification}
\fi
\makeatother

\ifpdf
\def\afs#1#2{\href{#1}{\nolinkurl{#2}}}
\else
\def\afs#1#2{\burlalt{#1}{#2}}
\fi

\def\affs#1#2{\href{#1}{\nolinkurl{#2}}}

\makeatother

\providecommand{\claimname}{Claim}
\providecommand{\definitionname}{Definition}
\providecommand{\remarkname}{Remark}
\providecommand{\theoremname}{Theorem}

\newcommand{\R}{\mathbb{R}}

\title[No Riesz basis]{A set with no Riesz basis of exponentials}
\author{Gady Kozma}
\address{GK: The Weizmann Institute of Science, Rehovot, Israel}
\email{gady.kozma@weizmann.ac.il}
\thanks{GK would like to acknowledge the support of the Jesselson Foundation.}

\author{Shahaf Nitzan}
\address{SN: Georgia Institute of Technology, Atlanta, USA}
\email{shahaf.nitzan@math.gatech.edu}
\thanks{SN is supported by NSF-CAREER grant DMS-1847796.}

\author{Alexander Olevski\u\i}
\address{AO: Tel Aviv University, Tel Aviv, Israel}
\email{olevskii@post.tau.ac.il}

\begin{document}

\begin{abstract}
  We show that there exists a bounded subset of $\R$ such that no system of exponentials can be a Riesz basis for the corresponding Hilbert space. An additional result gives a lower bound for the Riesz constant of any putative Riesz basis of the two dimensional disk.
\end{abstract}

\maketitle

\section{Introduction}

A system of vectors $\{u_{\lambda}\}$ in a separable Hilbert space
$H$ is called a basis if every vector $f\in H$ can be represented
by a series
\[
f=\sum c_{\lambda}u_{\lambda}
\]
and the representation is unique. The best kind of basis is the orthonormal
basis. In this paper we are interested in questions revolving around
the existence of bases when $u_{l}$ are taken from a specific, pre-given
set. More specifically, we are interested in the case that $H=L^{2}(S)$
for some $S\subset\mathbb{R}$ or $\mathbb{R}^{d}$, bounded and of
positive measure, and the $u_{\lambda}$ are exponential functions. In this
case it is not always possible to find an orthogonal basis:
If $S$ is an interval then the classic Fourier system is an orthogonal
basis. However, for the union of two intervals it is easy to see that
in general, no exponential orthogonal basis exists, say for
$[0,2]\cup[3,5]$. See \cite{L01} for a full treatment. So one needs
some generalisation of orthogonal bases, which would possess many of
their good properties, but would be more available for constructions.
\begin{defn*}
The image of an orthonormal basis by a linear isomorphism of the space
is called a Riesz basis. Equivalently, a system $\{u_{l}\}$ is Riesz
basis if:
\begin{enumerate}
\item It is complete in $H$;
\item There is a constant $K$ such that for any (finite) sum $P=\sum c_{l}u_{l}$
the following condition holds:
\begin{equation}
\frac{1}{K}||c||^{2}\le||P||^{2}\le K||c||^{2},\label{eq:def Riesz}
\end{equation}
 where $||c||^{2}=\sum |c_{l}|^{2}$.
\end{enumerate}
\end{defn*}
\begin{rem*}
One may ask: is any basis in $H$ a Riesz basis? Babenko \cite{B48}
gave a counterexample. His example was a system of exponentials in
a weighted $L^{2}$ space. Later, it turned out that the weights
defined by Babenko are a specific case of the so-called
Muckenhoupt weights \cite{M72} (we remark that the same condition was
discovered independently and simultaneously by Krantsberg, see
\cite{K72} or \cite[pg.\ 73]{O75}). 

Thus the question we are interested in is as follows. Given an $S\subset\mathbb{R}$,
is there a $\Lambda\subset\mathbb{R}$ such that the system $E(\Lambda)\coloneqq\{e^{2\pi i\lambda t}:\lambda\in\Lambda\}$
is a Riesz basis? It turns out that it is neither easy to construct
Riesz bases, nor to prove that none exist. For the construction problem,
Seip \cite{S95} constructed Riesz bases of exponentials for unions
of two intervals (and some cases of unions of larger numbers of intervals).
Riesz bases for arbitrary finite unions of intervals were constructed
in \cite{KN15}. See \cite{GL14,K15} for a construction of Riesz bases for
multitiling sets, and \cite{LS97} for exponentials with complex
frequencies. See \cite{arx1, arx2, L} for some recent work. Nevertheless, for many natural sets the question
is still open, with famous examples being the ball and triangle in
two dimensions.

In this paper we give an example of a set $S$ for which no Riesz basis
of exponentials exists. The same technique allows us to show that,
even if a Riesz basis of exponentials existed for a two-dimensional
ball, its defining constant $K$ (the $K$ from (\ref{eq:def Riesz}))
cannot be too close to 1. In particular this reproduces Fuglede's
result that the ball has no orthogonal basis of exponentials \cite{F74},
as that would correspond to $K=1$. For a different generalisation of
Fuglede's result, see \cite{IM}.
\end{rem*}

\section{Preliminaries}

Let us start by recalling the Paley-Wiener perturbation theorem.

\begin{thm*}
[Paley and Wiener] Let $S\subset\mathbb{R}^d$ be a bounded set of
positive measure, and let $\Lambda=\{\lambda_{n}\}\subset\mathbb{R}$
such that $E(\Lambda)$ is a Riesz basis for $L^{2}(S)$. Then there
exists a constant $\mu=\mu(S,\Lambda)$ such that if a second sequence
$\Gamma=\{\gamma_{n}\}$ satisfies $|\lambda_{n}-\gamma_{n}|<\mu$
for all $n$, then $E(\Gamma)$ is also a Riesz basis for $L^{2}(S)$.
%
\end{thm*}
See, e.g., \cite[\S 2.3]{KN15} for a proof in $d=1$. The proof in higher dimensions is similar.

We say that a set $\Lambda\subset\R$ is uniformly discrete if there
exists some $c>0$ such that if $\lambda\ne\mu$ are both in $\Lambda$
then $|\lambda-\mu|>c$. It is easy to see that any Riesz basis of
exponentials of 
$L^2(S)$ for some $S\subset\R^d$ is uniformly discrete. On the other
hand, any set which is uniformly discrete satisfies the right
inequality in (\ref{eq:def Riesz}), namely
\[
\bigg\Vert\sum_{\lambda\in \Lambda} c_\lambda e(\lambda
x)\bigg\Vert^2\le C\sum_{\lambda\in\Lambda}|c_\lambda|^2.
\]
(This inequality is called Bessel's inequality). See \cite[Proposition
  2.7]{OU16} --- the formulation here follows from the formulation in
the book by a simple duality argument.

Throughout we use the usual notation $e(x)=e^{2\pi ix}$ and $E(\Lambda)=\{e(\lambda x):\lambda\in\Lambda\}$. We use $c$
and $C$ for constants (usually depending on the
Riesz basis involved), whose
value may change from formula to formula and even inside the same
formula. We use $c$ for constants which are `small enough' and $C$
for constants which are `large enough'.

\section{Proofs}

We start with a result of the third named author which remained unpublished, but a version of it was included (with his permission) in the book of Heil \cite[pg.\ 296]{heil} (see also \cite{duklai} for a proof in the more general context of frames, or `overcomplete' bases). We chose to present it here as it provides the simplest demonstration of the `translation' technique which we apply throughout this note.
\begin{thm}
\label{thm:weighted}Let $S\subset\mathbb{R}$ and let $w\in L^1(S)$ be a positive function. If $w$ is not bounded away from $0$ or $\infty$, then there is
no $\Lambda\subset\mathbb{R}$ such that $E(\Lambda)$ is a Riesz
basis for the weighted space $L^{2}(S,w)$.
\end{thm}

\begin{proof}
Assume that $w$ is not bounded away from $0$. In this case one may find, for every $\varepsilon>0$,
a subset $A=A(\varepsilon)\subset S$ of positive measure and a number $t=t(\varepsilon)$
such that $w(x)<\varepsilon$ on $A$ but $w(x)>c$ on $A+t$, where
the constant $c$ does not depend on $\varepsilon$.

Examine the function $f=\frac{1}{\sqrt{|A|}}\mathbbm{1}_{A}$. Assume
by contradiction that $E(\Lambda)$ is a Riesz basis for  $L^{2}(S,w)$ with
respect to a constant $K$, and develop $f$ in this basis, i.e.\ write
\[
f(x)=\sum_{\lambda\in\Lambda}c_{\lambda}e(\lambda x).
\]
The Riesz basis property tells us that the series converges in $L^{2}$
and that
\begin{equation}
\sum|c_{\lambda}|^{2}\le K||f||_{L^{2}(S,w)}^{2}\le K\varepsilon,\label{eq:Riesz one}
\end{equation}
by the definition of $A$. Now perform a formal translation (by $-t$)
of the series $\sum c_{\lambda}e(\lambda x)$, namely, consider a
new series with coefficients $d_{\lambda}\coloneqq c_{\lambda}e(-\lambda t)$.
Since $\sum|d_{\lambda}|^{2}=\sum|c_{\lambda}|^{2}<\infty$, the Riesz
basis property gives that the series $\sum d_{\lambda}e(\lambda x)$
converges in  $L^{2}(S,w)$ and the limit, $g$, satisfies
\begin{equation}
||g||_{ L^{2}(S,w)}^{2}\le K\sum_{\lambda\in\Lambda}|d_{\lambda}|^{2}=K\sum_{\lambda\in\Lambda}|c_{\lambda}|^{2}\stackrel{\textrm{(\ref{eq:Riesz one})}}{\le}K^{2}\varepsilon.\label{eq:Riesz 2}
\end{equation}
On the other hand, for every $x\in A+t$ we have
\[
g(x)=\sum_{\lambda\in\Lambda}d_{\lambda}e(\lambda x)=\sum_{\lambda\in\Lambda}c_{\lambda}e(\lambda(x-t))=f(x-t)=\frac{1}{\sqrt{|A|}}
\]
(it is perhaps easiest to consider the equalities as holding almost
everywhere and the sums converging in measure --- we use here that
convergence in $L^{2}$ implies convergence in measure; and that if
a sum converges both in $L^{2}$ and in measure then the two limits
are almost everywhere equal). But this is a contradiction because
then
\[
||g||_{L^{2}(s,w)}\ge\int_{A+t}|g(x)|^{2}w(x)\,dx\ge\int_{A+t}\frac{1}{|A|}c\,dx=c,
\]
contradicting (\ref{eq:Riesz 2}), since $\varepsilon$ was arbitrary.

The case that $w$ is not bounded away from $\infty$ is treated in a similar way.
\end{proof}

Next, we turn to the main result of this note.

\begin{thm}
There exists a bounded set $S\subset\mathbb{R}$ for which no
$\Lambda\subset\mathbb{R}$ may satisfy that $E(\Lambda)$ is a Riesz
basis for $L^{2}(S)$ 
\end{thm}

It might be interesting to remark already at this point that the set
is an infinite collection of intervals with one accumulation point. 

For the proof, we will use the following auxiliary claim.
\begin{claim*}
Let $A\subset\R$ be a set of positive finite measure, and let $\{g_i\}_{i= 1}^N\subset L^2(A)$. If for every $U\subset\{1,\dotsc,N\}$ we have
\begin{equation}\label{linear comb}
\bigg|\sum_{i\in U}g_i(x) \bigg|\leq M\sqrt{|U|},\qquad \forall x\in A,
\end{equation}
then there exists some $i_{0}$ such that
\[
\int_{A}|g_{i_{0}}|^{2}\le\frac{6M^2|A|}{\sqrt{N}}.
\]
\end{claim*}
\begin{proof}
Assume that $\int_{A}|g_{i}|^{2}>\delta$ for some $\delta$ and all
$i$. We get
\[
\sum_{i}\int_{A}|g_{i}(x)|^{2}\,dx>\delta N.
\]
Hence there exists some $x_{0}$ such that
\[
\sum_{i}|g_{i}(x_{0})|^{2}>\frac{\delta N}{|A|}
\]
Condition (\ref{linear comb}) implies in particular that
$
|g_{i}(x_{0})|\le M
$
for all $i$ 
and so
\[
\sum_{i}|g_{i}(x_{0})|\geq \frac{1}{M}\sum_{i}|g_{i}(x_{0})|^2>\frac{\delta N}{|A|M}.
\]
Let $U_{1}$, $U_{2}$, $U_{3}$ be defined by
\[
U_{j}=\left\{ i:\arg g_{i}(x_{0})\in \left[2\pi\frac{j-1}{3},2\pi\frac{j}{3}\right)\right\} .
\]
Hence for some $j$,
\[
\sum_{i\in U_{j}}|g_{i}(x_{0})|>\frac{\delta N}{3|A|M}.
\]
Denote $U\coloneqq U_{j}$ for this $j$ for short. If a collection
of complex numbers $c_{i}$ is in a sector of opening $120^{\circ}$
then $|\sum c_{i}|\ge\frac{1}{2}\sum|c_{i}|$ and hence
\[
\bigg|\sum_{i\in U}g_{i}(x_{0})\bigg|\ge\frac{\delta N}{6|A|M}.
\]
Since $|U|\leq N$, combining this with (\ref{linear comb}) gives the required result.
\end{proof}

We are now ready to prove Theorem 2.

\begin{proof} 
We start with $S_{1}=[0,1]\cup[2,3]$. We then break the interval
$[2,3]$ into $4^{4^{2}}+1$ intervals of equal length, and keep for
each one only the left half (this is $S_{2}$). We next take the last
interval of $S_{2}$, break it into $4^{4^{3}}+1$ intervals of equal
length, and keep for each one only the left third. We continue this
way and denote $S=\bigcap S_{i}$.

Assume by contradiction that $S$ has a Riesz basis $E(\Lambda)$ of
exponentials. By the Paley-Wiener theorem we may assume without loss
of generality that $\Lambda\subset\frac{1}{\ell}\mathbb{Z}$ for some
integer $\ell\ge 4$. Let $K$
be a Riesz constant for $\Lambda$, i.e.\ every $f\in L^{2}(S)$ can be expanded
into a sum
\[
f(x)=\sum_{\lambda\in\Lambda}c_{\lambda}e(\lambda x)
\]
such that
\[
\frac{1}{K}\sum|c_{\lambda}|^{2}\le\int_{S}|f|^{2}\le K\sum|c_{\lambda}|^{2}.
\]
Assume for simplicity that $K\ge\ell$.

Fix some $n$. The construction of $S$ implies that one may find
$4^{4^{n}}$ intervals $I_{1},\dotsc,I_{4^{4^{n}}}$ of equal length
(call it $\epsilon$) and distance between them $\epsilon(n-1)$.
For $i=1,\dotsc,4^{4^n}$ let
\[
f_{i}=\frac{1}{\sqrt{\epsilon}}\mathbbm{1}_{I_{i}}
\]
so that $||f_{i}||_{L^{2}(S)}=1$. Expand $f_{i}(x)=\sum c_{\lambda,i}e(\lambda x)$.
Since $\Lambda\subset\frac{1}{\ell}\mathbb{Z}$, this sum converges in $L^{2}[0,\ell]$ and the limit is an extension
of $f$ to the whole of $L^{2}[0,\ell]$ (to the whole of $\mathbb{R}$,
if you prefer). Call these extensions $\widetilde{f_{i}}$ and note that
the Riesz basis property gives that $\sum|c_{\lambda,i}|^{2}\le K$
and hence
\[
||\widetilde{f_{i}}||_{L^{2}([0,\ell])}^{2}\le \ell K\le K^{2}.
\]
A similar argument shows that for any coefficients $\beta_{i}$
\begin{equation}
\bigg\Vert\sum_{i}\beta_{i}\widetilde{f_{i}}\bigg\Vert_{L^{2}([0,\ell])}^{2}\le K^{2}\sum|\beta_{i}|^{2}.\label{eq:sum bi fi}
\end{equation}
Let us further define
\[
h=\frac{4}{\epsilon}\mathbbm{1}_{[-\epsilon/8,\epsilon/8]}
\]
so that $||h||_{1}=1$. Let $g_{i}=\widetilde{f_{i}}*h$. By the Cauchy-Schwarz inequality, for every $U\subset\{1,\dotsc,4^{4^n}\}$ we have
\[
\bigg|\sum_{i\in U}g_{i}(x)\bigg|
\le\bigg\Vert\sum_{i\in U}\widetilde{f_{i}}\bigg\Vert_{L^2[0,\ell]}\Vert h\Vert_{2}
\stackrel{\textrm{(\ref{eq:sum bi fi})}}{\le}
K\sqrt{|U|}\cdot\frac{2}{\sqrt{\epsilon}}\qquad\forall x\in\R.
\]
Hence the functions $g_i$ satisfy condition (\ref{linear comb}) with $M=2K/\sqrt{\epsilon}$.

Further, enumerate all intervals comprising $S$ which have length
strictly bigger than $\epsilon$ as $[a_{1},b_{1}],\dotsc,[a_{e},b_{e}]$ where
$e=\sum_{k<n}4^{4^{k}}\le2\cdot4^{4^{n-1}}$. Let
\[
A\coloneqq\bigcup_{i=1}^{e}[a_{i}-n\epsilon,a_{i}+\epsilon].
\]
($A$ is the set of left edges of these intervals). Then $|A|=e(n+1)\epsilon\leq 2(n+1)\epsilon4^{4^{n-1}}$. By the auxiliary claim (with $N=4^{4^n}$ and $M$ as above) there exists an $i_{0}$ such that
\[
\int_{A}|g_{i_{0}}|^{2}\le\frac{6M^2|A|}{\sqrt{N}}\le \frac{CK^{2}n}{4^{4^{n-1}}}.
\]
With $i_0$ selected and this crucial property of $g_{i_0}$ proved, the theorem follows by examining a few translations of $g_{i_0}$.

To see this, expand $f_{i_0}$ using the Riesz
basis property and denote the coefficients by $c_{\lambda}$ so that
$f_{i_{0}}(x)=\sum c_{\lambda}e(\lambda x)$, and recall that we have
extended $f_{i_{0}}$ so that a similar equality holds also for $\widetilde{f_{i_0}}$ over all of $\mathbb{R}$. Hence
\[
g_{i_{0}}(x)=\sum_{\lambda\in\Lambda}c_{\lambda}\widehat{h}(\lambda)e(\lambda x).
\]
So $g_{i_{0}}$ is supported spectrally on $\Lambda$ and therefore
\begin{equation}
\sum|c_{\lambda}\widehat{h}(\lambda)|^{2}\ge\frac{1}{K}\int_{S}|g_{i_{0}}|^{2}\ge\frac{1}{2K},\label{eq:ah large}
\end{equation}
where the last inequality follows from the definition of $g_{i_0}$ as
$\widetilde{f}_{i_0}*h$ which allows to calculate $g_{i_0}$ in the
middle half of the interval $I_{i_0}$. On the other hand, the Plancharel formula on $[0,\ell]$ and the Riesz
property imply that
\begin{equation}
\int_{0}^{\ell}|g_{i_{0}}|^{2}=\ell\sum|c_{\lambda}|^{2}|\widehat{h}(\lambda)|^{2}\le K\sum|c_{\lambda}|^{2}\le K^{2}\int_{S}|f_{i_0}|^{2}=K^{2}.\label{eq:gi0 small}
\end{equation}
Consider the following $n-1$ translations of $g_{i_{0}}$:
\[
G_{k}(x)=g_{i_{0}}(x-k\epsilon)=\sum_{\lambda\in\Lambda}c_{\lambda}\widehat{h}(\lambda)e(-\lambda k\epsilon)e(\lambda x).
\]
Each $G_{k}$ is also supported spectrally on $\Lambda$ and hence
\[
\int_{S}|G_{k}|^{2}\ge\frac{1}{K}\sum|c_{\lambda}\widehat{h}(\lambda)|^{2}\stackrel{\textrm{(\ref{eq:ah large})}}{\ge}\frac{1}{2K^{2}}.
\]
Consider now an interval $J$ whose length is larger than $\epsilon$. Since $\widetilde{f}_{i_0}$ is zero on $J$, we have that $g_{i_0}$ is zero except on the $\frac18\epsilon$ edges of $J$. Hence $G_{k}|_{J}$, $k\ge 1$ is zero on most of the interval, except for its left end, where $G_{k}$ is a translation of $g_{i_{0}}$
on some part of $A$. Therefore
\[
\sum_{J:|J|>\epsilon}\int_{J}|G_{k}|^{2}\le\int_{A}|g_{i_{0}}|^{2}\le\frac{CK^{2}n}{4^{4^{n-1}}}.
\]
This means that, if $I$ is the union of all intervals of length $\epsilon$
or less (all the $I_{i}$ and then the shorter intervals) then
\[
\int_{I}|G_{k}|^{2}\ge\frac{1}{2K^{2}}-\frac{CK^{2}n}{4^{4^{n-1}}}
\]
and since $G_{k}$ is a translation of $g_{i_{0}}$ we get
\[
\int_{I-\epsilon k}|g_{i_{0}}|^{2}\ge\frac{1}{2K^{2}}-\frac{CK^{2}n}{4^{4^{n-1}}}.
\]
Since the sets $I-\epsilon k$ are disjoint, for $k=1,\dotsc,n-1$,
we get
\[
K^{2}\stackrel{\textrm{(\ref{eq:gi0 small})}}{\ge}
\int_{0}^{\ell}|g_{i_{0}}|^{2}
\ge(n-1)\left(\frac{1}{2K^{2}}-\frac{CK^{2}n}{4^{4^{n-1}}}\right).
\]
Since $n$ was arbitrary, we have reached a contradiction.
\end{proof}
Our last result is that, even if a Riesz basis of exponentials would
exist for the disk, its constant cannot be too close to 1. The way
we defined the Riesz constant in (\ref{eq:def Riesz}), though, makes
it easy to compare to orthonormal bases only when $|S|=1$ (otherwise
even an orthogonal basis of exponentials would not be normalised and
would not satisfy (\ref{eq:def Riesz}) with $K=1$). For simplicity we
let $D=\frac{1}{\sqrt{\pi}}\mathbb{D}$ i.e.\ a disk with area 1, and
prove the result for it. 

\begin{thm}
Any Riesz basis of exponentials for $D$ must have $K\ge\sqrt{\frac{1+\sqrt{5}}{2}}$.
\end{thm}

\begin{proof}
Assume by contradiction that $\Lambda\subset\mathbb{R}^2$ satisfies
that $E(\Lambda)$ is a Riesz basis for $L^{2}(D)$, with the Riesz
constant $K$ (the $K$ from (\ref{eq:def Riesz})) satisfying $K<\sqrt{\frac{1+\sqrt{5}}{2}}$.

Fix $\varepsilon>0$ and examine the function
\[
f=\frac{1}{\sqrt{\pi}\varepsilon}\mathbbm{1}_{\varepsilon\mathbb{D}}
\]
i.e.\ a normalised indicator of a disk of radius $\varepsilon$ around
0. Use the Riesz basis property to write
\[
f(x)=\sum_{\lambda\in\Lambda}c_{\lambda}e(\langle\lambda,x\rangle)
\]
with the sum converging in $L^{2}(D)$. As before, we use this representation
to extend $f$ to $[-2,2]^2$, though there will be two differences from
the previous theorem. First, it will be more convenient to not
distinguish between $f$ and its extension, and call the extension $f$
as well. And second, rather than using Paley-Wiener, this time we
use the fact $E(\Lambda)$ is a Riesz basis in $L^2(D)$ to conclude
that $\Lambda$ is uniformly separated and hence $E(\Lambda)$ satisfies Bessel's inequality
in $L^2([-2,2]^2)$, i.e 
\begin{equation}
||f||_{L^{2}([-2,2]^{2})}\le C\sum_{\lambda\in\Lambda}|c_{\lambda}|^{2}\le CK||f||_{D}^{2}=CK.\label{eq:f<=00003DN2K}
\end{equation}
For a $\theta\in[0,2\pi]$ denote $t_{\theta}=(\nicefrac{1}{\sqrt{\pi}}-\varepsilon)(\cos\theta,\sin\theta)$
and consider the translation of $f$ by $t_{\theta}$ i.e.
\[
g_{\theta}(x)\coloneqq\sum_{\lambda\in\Lambda}c_{\lambda}e(\langle\lambda,x-t_{\theta}\rangle).
\]
As before, we must have
\[
||g_{\theta}||_{L^{2}(D)}^{2}\le K\sum_{\lambda\in\Lambda}|c_{\lambda}e(\langle\lambda,-t_{\theta}\rangle)|^{2}=K\sum_{\lambda\in\Lambda}|c_{\lambda}|^{2}\le K^{2}.
\]
However, $g_{\theta}(x)=\frac{1}{\sqrt{\pi}\varepsilon}$ for a disk
of radius $\varepsilon$ contained in $D$. Hence
\[
\int_{D\setminus(D+t_{\theta})}|g_{\theta}|^{2}\le K^{2}-1.
\]
Translating back we get an estimate for (the extended) $f$ outside
$D$, namely
\begin{equation}
\int_{(D-t_{\theta})\setminus D}|f|^{2}\le K^{2}-1.\label{eq:good translation}
\end{equation}
This is our upper bound for $f$.

To get a lower bound we again examine one $\theta$, and this time
translate by $s_{\theta}=(\nicefrac{1}{\sqrt{\pi}}+\varepsilon)(\cos\theta,\sin\theta)$.
Denote the translated function by $h_{\theta}$, namely
\[
h_{\theta}(x)\coloneqq\sum_{\lambda\in\Lambda}c_{\lambda}e(\langle\lambda,x-s_{\theta}\rangle)
\]
and get
\[
||h_{\theta}||_{L^{2}(D)}^{2}\ge\frac{1}{K}\sum_{\lambda\in\Lambda}|c_{\lambda}|^{2}\ge\frac{1}{K^{2}}.
\]
which we again map back to $f$ to get
\[
\int_{(D-s_{\theta})\setminus D}|f|^{2}\ge\frac{1}{K^{2}}
\]
(note that we used in this last step that the little disk that allowed
to subtract 1 when we did the corresponding calculations for $g$ is
outside $D$ for $h$, so does not contribute to $||h||_{L^2(D)}$). With the lower bound (\ref{eq:good translation})
we get
\[
\int_{(D-s_{\theta})\setminus(D-t_{\theta})}|f|^{2}\ge\frac{1}{K^{2}}-(K^{2}-1).
\]
When $K<\sqrt{\frac{1+\sqrt{5}}{2}}$ we get that the right-hand side
is at least some constant $c>0$. Integrating over $\theta$ and
using Fubini gives (omitting the details of the elementary geometry
exercise involved)
\[
  c\le \int_{0}^{2\pi}\int_{(D-s_{\theta})\setminus(D-t_{\theta})}|f(x)|^{2}\,dx\,d\theta
\le \int_{(2/\sqrt{\pi}+\varepsilon)\mathbbm{D}}C\sqrt{\varepsilon}|f(x)|^{2}\,dx
\]
Or, in other words,
\[
\int_{(2/\sqrt{\pi}+\varepsilon)\mathbbm{D}}|f(x)|^{2}\ge\frac{c}{\sqrt{\varepsilon}}
\]
contradicting (\ref{eq:f<=00003DN2K}) if $\varepsilon$ is taken to be
sufficiently small.
\end{proof}

\end{document}